\documentclass{article}

\usepackage{amsmath, amssymb}
\usepackage{graphicx, wrapfig} 
\usepackage{amsthm}
\usepackage{url}
\usepackage{color}

\theoremstyle{plain}
\newtheorem{theorem}{Theorem}
\newtheorem{lemma}{Lemma}

\theoremstyle{definition}
\newtheorem{definition}{Definition}
\newtheorem{example}{Example}

\newtheorem{proposition}{Proposition}
\newtheorem{corollary}{Corollary}

\newtheorem{conjecture}{Conjecture}

\title{Characterization of the OU matrix of a braid diagram}
\author{Ayaka Shimizu\thanks{Osaka Central Advanced Mathematical Institute, Osaka Metropolitan University, Sugimoto, Osaka, 558-8585, Japan. Email: shimizu1984@gmail.com}
and Yoshiro Yaguchi\thanks{Maebashi Institute of Technology, Maebashi, Gunma, 371-0816, Japan. Email: y.yaguchi@maebashi-it.ac.jp}}
\date{\today}

\begin{document}

\maketitle

\begin{abstract}
The OU matrix of a braid diagram is a square matrix that represents the number of over/under crossings of each pair of strands. 
In this paper, the OU matrix of a pure braid diagram is characterized for up to 5 strands. 
As an application, the crossing matrix of a positive pure braid is also characterized for up to 5 strands. 
Moreover, a standard form of the OU matrix is given and characterized for general braids of up to 5 strands. 
\end{abstract}

\section{Introduction}
\label{section-intro}

An {\it $n$-braid} is an object in $\mathbb{R}^3$ consisting of $n$ strands whose endpoints are fixed at horizontal bars at the top and bottom, where each strand runs from top to bottom without returning (\cite{Artin}). 
We say that two $n$-braids are {\it equivalent} (or {\it same}) if they are sent to each other by an ambient isotopy of ${\mathbb R}^3$ fixing the top and bottom bars pointwise. 
In this paper, we call the equivalence class of a braid simply a braid. 
An {\it $n$-braid diagram} is a regular projection of an $n$-braid with horizontal bars in which each intersection point has the over/under information as shown in Figure \ref{fig-5-braid}. 
\begin{figure}[ht]
\centering
\includegraphics[width=3.5cm]{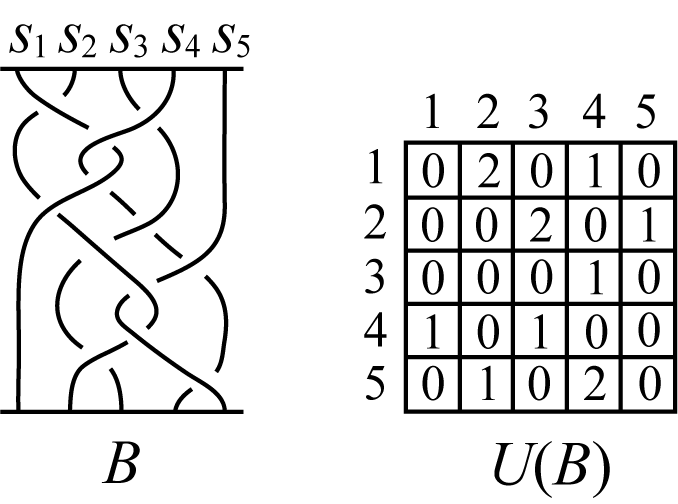}
\caption{A 5-braid diagram $B$ and the OU matrix $U(B)$.}
\label{fig-5-braid}
\end{figure}

Let $s_1, s_2, \dots , s_n$ be the strands of an $n$-braid diagram $B$  with subscripts from the left-hand side to the right-hand side on the top of the braid diagram as shown in Figure \ref{fig-5-braid}. 
The {\it OU matrix}, $U(B)$, of $B$ is an $n \times n$ matrix with zero diagonal such that the $(i,j)$ entry is the number of crossings between $s_i$ and $s_j$ where $s_i$ is over $s_j$ as shown in Figure \ref{fig-5-braid}. 
(For more precise definition, see \cite{AY} or Section \ref{section-preliminaries}.)
The OU matrix was introduced in \cite{AY} to discuss the ``warping degree'' of a braid diagram (\cite{ASA}), which estimates the unknotting number of the closure of a braid diagram (see \cite{Ka, Al} for the warping degree of a link diagram). 
It was also found that the determinant of the OU matrix acts effectively on layered braid diagrams (\cite{AY}). 
In this paper, we characterize the OU matrix of a pure braid diagram of up to $5 \times 5$. \par 

We call a matrix whose entries are all integers (resp. non-negative integers) an {\it integer matrix} (resp. {\it non-negative integer matrix}). 
We call a matrix whose entries are all non-negative even integers a {\it non-negative even matrix}. 
We say that a square matrix $M$ is a {\it zero-diagonal matrix} when $M$ has zero diagonal, namely all the $(i,i)$ entries on the main diagonal are zero. 
We denote the $(i,j)$ entry of a matrix $M$ by $M(i,j)$. 
The following properties of $n \times n$ zero-diagonal matrices, T0 and T1, are introduced in \cite{Bu}.

\begin{definition}[\cite{Bu}]
A zero-diagonal matrix $M$ is said to be T0 if whenever $1 \leq i < j < k \leq n$, then $M(i,j)=M(j,k)=0$ implies $M(i,k)=0$. A zero-diagonal matrix $M$ is said to be T1 if whenever $1 \leq i < j < k \leq n$, then $M(i,j),~M(j,k) \neq 0$ implies $M(i,k) \neq 0$. 
\end{definition}

\noindent For example, the matrices 
\begin{align*}
A=
\begin{pmatrix}
0 & 0 & 1 & 0 \\
0 & 0 & 0 & 1 \\
1 & 0 & 0 & 1 \\
0 & 1 & 1 & 0 
\end{pmatrix} 
, \ 
B=
\begin{pmatrix}
0 & 0 & 1 & 2 \\
0 & 0 & 1 & 0 \\
1 & 1 & 0 & 1 \\
2 & 0 & 1 & 0 
\end{pmatrix} 
\end{align*}
are not T0 since they do not meet the condition with $A(1,2)=A(2,3)=0$ and $A(1,3)=1$, $B(1,2)=B(2,4)=0$ and $B(1,4)=2$. 
In this paper, when we say that a matrix $M$ is T0, it means that $M$ is a T0 zero-diagonal matrix. 
Let $M^T$ denote the transpose of a matrix $M$. 
In this paper, we prove the following theorem. 

\begin{theorem}
When $n \leq 5$, an $n \times n$ matrix $M$ is the OU matrix of some pure $n$-braid diagram if and only if $M+M^T$ is a non-negative even T0 matrix. 
\label{thm-OU5}
\end{theorem}

For an $n$-braid diagram $B$, the {\it crossing matrix} $C(B)$ of $B$ is defined in \cite{Bu} to be the $n \times n$ zero-diagonal matrix whose $(i,j)$ entry is the algebraic number (positive minus negative) of crossings of $B$ where the strand $s_i$ is over the strand $s_j$ (see also \cite{Gu}).
If two braid diagrams $B$ and $B'$ represent the same braid, then $C(B)=C(B')$ holds by definition, and therefore the crossing matrix is an invariant of equivalence classes of braids. 
This means that the crossing matrix $C([B])$ of the equivalence class $[B]$ of $B$ is well defined. 
We denote $C([B])$ by $C(B)$ for simplicity. 
In \cite{Bu}, it is proved that the crossing matrix can completely characterize positive braids of canonical length at most 2 (in terms of the canonical form as explained in \cite{El, ThB}), whereas it is also proved that it is impossible for length 3 or more.  \par
As for the OU matrices, the equation $U(B)=U(B')$ does not hold for the braid diagrams $B$ and $B'$ that represent the same braid because the OU matrix changes when two consecutive crossings of two strands are produced or reduced, that is, it changes by a Reidemeister move of type I\hspace{-0.5pt}I. 
This implies that the OU matrix is not an invariant of equivalence classes of braids but an invariant of (planner isotopy class of) braid diagrams. 
For positive braids, it is shown in \cite{AY} that the OU matrix is an invariant by taking any positive diagram. \par 
In this paper, we show that the equality $C(B)=U(B)$ holds for any positive braid diagram $B$ (Proposition \ref{prop-OU-C}) and hence for any positive braids as well. 
In \cite{Bu}, a conjecture is proposed about a necessary and sufficient condition for a matrix to be the crossing matrix $C(B)$ of a positive pure braid $B$, and it was proved that the conjecture is true when $n\leq 3$. 
In this paper, we show that the conjecture is true when $n \leq 5$ (Corollary \ref{cor-C1}) by discussing OU matrices. 
The conjecture is now open for $n\geq 6$, although an algorithm that exhibits all positive braids with a given crossing matrix, if any exist, or declares that there are none is known (see \cite{Gu}).

The rest of the paper is organized as follows. 
In Section \ref{section-preliminaries}, we describe the definitions and properties of the OU matrix, CN matrix, and the crossing matrix. 
In Section \ref{section-02pure}, we investigate CN matrices whose entries are 0 or 2, introducing the BW-ladder diagram. 
In Section \ref{section-5b}, we investigate the CN matrices of $n$-pure braid diagrams for $n \leq 5$. 
In Section \ref{section-pf}, we prove Theorem \ref{thm-OU5}, namely, we characterize the OU matrices of pure braid diagrams up to $5 \times 5$. 
We also discuss the crossing matrices of positive pure braids and the OU matrices of braid diagrams that are not necessarily pure.

\section{OU, CN and crossing matrices}
\label{section-preliminaries}

This section includes an introductory overwiew of the OU, CN and crossing matrices. 
In Section \ref{subs-OU}, we review the definition of the OU matrix of a braid diagram which was defined in \cite{AY}. 
In Section \ref{subs-CN}, we define the CN matrix which will be used for the proof of the main theorem and see some properties. 
In Section \ref{subs-crossing}, we review the definition of the crossing matrix of a braid which was defined in \cite{Bu}. 
In Section \ref{subs-OUC}, we discuss the relation between the OU matrices and the crossing matrices. 
In Section \ref{subs-adm}, we discuss matrices which realize such matrices.

\subsection{OU matrix}
\label{subs-OU}

In this subsection, we review the definition and some properties of the OU matrix defined in \cite{AY}. 
Let $B$ be an $n$-braid diagram with the strands $s_1, s_2, \dots , s_n$ that are ordered from the left-hand side to the right-hand side on the top of $B$. 
Let $\pi$ be a permutation on $(1,2, \dots ,n)$. 
The {\it OU matrix $U(B, \pi )$ of $B$ with $\pi$} is an $n \times n$ zero-diagonal matrix $M$ such that $M(i,j)$ is the number of crossings on the $i^{th}$ strand that are over the $j^{th}$ strand ($i \neq j$), where the ``$k^{th}$ strand'' stands for the strand $s_{\pi(k)}$, that is, the $k^{th}$ strand according to the order $\pi$. 
In this paper, we also denote $U(B, id)$ simply by $U(B)$ for the identity permutation $id$. 
In Figure \ref{fig-ex-OUCN}, the OU matrices of a braid diagram $B_1$ with the permutations $id$ and $\pi =(4,3,1,2)$ are shown as $U(B_1)$ and $U(B_1 , \pi )$, respectively. 
\begin{figure}[ht]
\centering
\includegraphics[width=12cm]{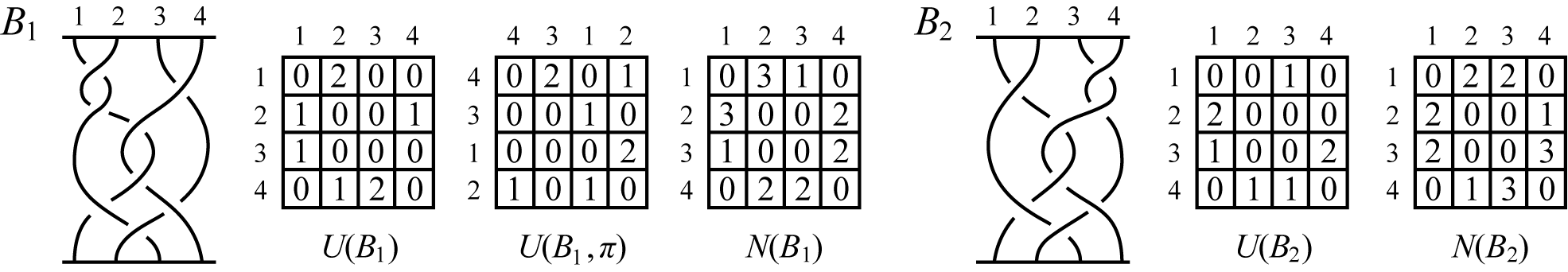}
\caption{Braid diagrams $B_1$, $B_2$ and their OU, CN matrices.}
\label{fig-ex-OUCN}
\end{figure}

\noindent For $n$-braid diagrams $B$ and $C$, we define the product $BC$ by the $n$-braid diagram obtained by placing $B$ above $C$ so that the endpoints of $B$ on the bottom bar coincide with those of $C$ on the top bar. 
We call this product the {\it braid product}. 
We have the following proposition. 

\begin{proposition}[\cite{AY}]
Let $B$, $C$ be $n$-braid diagrams, and let $\rho_B$ be the braid permutation of $B$. Then $U(BC, \pi)=U(B, \pi)+U(C, \pi \rho_B)$ holds for any strand permutation $\pi$. 
In particular, when $B$ is a pure braid, namely $\rho_B =id$, we have $U(BC, \pi)=U(B, \pi)+U(C, \pi)$.
\label{prop-BC-OU}
\end{proposition}

\subsection{CN matrix}
\label{subs-CN}

In this subsection, we define the CN matrix. 
Let $B$ be an $n$-braid diagram with the strands $s_1, s_2, \dots , s_n$ ordered from left to right on the top of $B$. 
A {\it crossing number matrix} $N(B)$, or {\it CN matrix}, of an $n$-braid diagram $B$ is an $n \times n$ zero-diagonal matrix $M$ such that $M(i,j)$ is the number of crossings between $s_i$ and $s_j$ ($i \neq j$). 
By definition, $N(B)$ is a symmetric matrix, and $N(B)=U(B)+ \left( U(B) \right)^T$ holds for any $B$ (see Figure \ref{fig-ex-OUCN}). 
From Proposition \ref{prop-BC-OU}, we have the following proposition. 

\begin{proposition}
The addition $N(BC)=N(B)+N(C)$ holds when $B$ is a pure braid diagram. 
\label{prop-N-sum}
\end{proposition}

\noindent We also have the following proposition for pure braid diagrams. 

\begin{proposition}
All the entries of the CN matrix $N(B)$ of a braid diagram $B$ are even numbers if and only if $B$ is a pure braid diagram. 
\label{prop-even-pure}
\end{proposition}

\begin{proof}
A braid diagram $B$ has the identity braid permutation if and only if each pair of strands has the same relative position on the top and bottom, and that is equivalent to that each pair of strands has an even number of crossings between them.
\end{proof}

\begin{definition}
For a square matrix $M$, the {\it reverse of $M$}, denoted by $M'$, is the matrix obtained from $M$ by reversing the order of the row and column. 
\end{definition}

\begin{proposition}
If a matrix $M$ is the CN matrix of a braid diagram, then the reverse $M'$ is also the CN matrix of a braid diagram.
\label{prop-order-rev}
\end{proposition}

\begin{proof}
Let $B$ be an $n$-braid diagram such that $N(B)=M$. 
Take the vertical rotation $B'$. 
Then, the order of the strands is reversed, and we have $N(B')=M'$. (See $N(B_1)$ and $N(B_2)$ in Figure \ref{fig-ex-OUCN}, where $B_2 =B_1'$ and $N(B_2)=(N(B_1))'$.)
\end{proof}

\begin{proposition}
If $M$ is the CN matrix of a braid diagram, then $M$ is T0. 
\label{prop-adequate}
\end{proposition}

\begin{proof}
Let $B$ be a braid diagram such that $N(B)=M$. 
If $M(i,j) \neq 0$ for a pair of $i$ and $j$ with $i<j$, then $B$ has a crossing between the $i^{th}$ and $j^{th}$ strands by definition. 
Let $c$ be the upper-most crossing between the $i^{th}$ and $j^{th}$ strands. 
Let $R$ be the region bounded by the top horizontal bar, the $i^{th}$ and $j^{th}$ strands from the top to $c$. 
Since $R$ is homeomorphic to the disc $D^2$, each $k^{th}$ strand with $i<k<j$ has a positive even number of intersections with the boundary $\partial R$, and exactly one of them is on the top bar and the others are on the $i^{th}$ or $j^{th}$ strands. 
Hence, the $k^{th}$ strand has a crossing with the $i^{th}$ or $j^{th}$ strand, and we have $M(i,k)>0$ or $M(k,j)>0$. 
\end{proof}

\noindent Since CN matrices are zero-diagonal symmetric matrices, we consider strictly upper triangular matrices instead of symmetric matrices in Sections \ref{section-02pure} and \ref{section-5b} for simplicity.

\subsection{Crossing matrix}
\label{subs-crossing}

As mentioned in Section \ref{section-intro}, the crossing matrix is a matrix defined in \cite{Bu} for braids. 
Let $B$ be an $n$-braid diagram with strands $s_1, s_2, \dots , s_n$ with subscripts from left to right at the top of the braid diagram. 
The crossing matrix $C(B)$ of $B$ is an $n \times n$ zero-diagonal matrix such that each $(i,j)$ entry is the algebraic number of crossings, i.e., the number of positive crossings minus the number of negative crossings\footnote{The definition of positive/negative crossings in this paper is opposite to \cite{Bu} and same to \cite{Gu} or \cite{AY}.}, of $B$ where $s_i$ is over $s_j$. 
For a braid diagram $B$, we denote $[B]$ by the equivalent class of a braid that has $B$ as a diagram. 
We note that if two diagrams $B$ and $B'$ represent the same braid in ${\mathbb R}^3$, i.e., $[B]=[B']$, then $C(B)=C(B')$ (see [2, 5]). 
Therefore, for a braid $b=[B]$ that has a diagram $B$, we can define the crossing matrix $C(b)$ of $b$ by $C(B)$. 
In \cite{Bu}, the crossing matrices of $n$-braids and pure $n$-braids are completely characterized for all $n \in \mathbb{N}$. 
More specifically, it is proved in \cite{Bu} that an $n\times n$ matrix $M$ is the crossing matrix of a pure braid if and only if $M$ is a zero-diagonal integer symmetric matrix.
For positive pure braids, it is conjectured as follows. 

\begin{conjecture}[\cite{Bu}]
An $n\times n$ matrix $M$ is the crossing matrix of some positive pure braid if and only if $M$ is a non-negative integer T0 symmetric matrix.
\label{conj-C}
\end{conjecture}

\noindent It has been shown that the conjecture is true when $n \leq 3$ in \cite{Bu}.

\subsection{OU matrix and crossing matrix}
\label{subs-OUC}

In this subsection, we discuss the relation between the OU matrix and the crossing matrix. 
In \cite{AY}, the following proposition was shown for positive pure braid diagrams. 

\begin{proposition}[\cite{AY}]
The OU matrix of a positive pure braid diagram is a symmetric matrix. 
\label{prop-sy}
\end{proposition}

\noindent The OU matrix coincides with the crossing matrix for positive braid diagrams. 

\begin{proposition}
When $B$ is a positive braid diagram, the equality $OU(B)=C(B)$ holds. 
\label{prop-OU-C}
\end{proposition}

\begin{proof}
Let $k$ be the number of crossings between a pair of strands $s_i$ and $s_j$ where $s_i$ is over $s_j$. 
By the definition of the OU matrix, the $(i,j)$ enrty of the OU matrix $U(B)$ is $k$. 
The $(i,j)$ entry of the crossing matrix $C(B)$ is also $k$ because all the crossings are positive crossings. 
\end{proof}

Let ${\mathcal B}_n$ be the set of planar isotopy classes of $n$-braid diagrams. 
Then, ${\mathcal B}_n$ forms a monoid under the braid product. 
Let ${\mathcal P}_n$  be the submonoid of ${\mathcal B}_n$ which consists of planar isotopy classes of pure braid diagrams and let ${\mathcal P}_n^{+}$ be the submonoid of ${\mathcal P}_n$ which consists of planar isotopy classes of positive pure braid diagrams. 
Let $B_n$ be the $n$-braid group, that is, the set of equivalent classes of $n$-braids with the group operation naturally induced by the braid product. 
Let $P_n$ be the subgroup of $B_n$ which consists of pure $n$-braids and let $P_n^{+}$ be the monoid which consists of positive pure $n$-braids. \par 

Let $X_n$ be the set of $n\times n$ zero-diagonal integer matrix and let $X_n^{+}$ be the set of matrices $M\in X_n$ such that $M(i,j) \geq 0$ for any $i,j$. 
We note that $(X_n,+)$ is a group whereas $(X_n^{+},+)$ is a monoid. 
Let $C: B_n\to X_n$ be the map defined by that $C(b)=C(B)$ is the crossing matrix of $b=[B]\in B_n$, let $U: {\mathcal B}_n\to X_n^{+}$ be the map defined by that $U(B)$ is the OU matrix of $B\in {\mathcal B}_n$ and let $N: {\mathcal B}_n\to X_n^{+}$ be the map defined by $N(B)$ is the CN matrix of $B\in {\mathcal B}_n$. \par 

For $M\in X_n$ and a permutation $\pi \in S_n=Sym\{1,\dots ,n\}$, let $M^{\pi}$ be the $n\times n$ matrix obtained by rearranging the rows and columns of $M$ according to $\pi$, i.e. the $i^{th}$ row and column are $\pi^{-1}(i)^{th}$ row and column of the original $M$, respectively. 
We easily see that for $M\in X_n$ (resp. $M\in X_n^{+}$), it holds that $M^{\pi}\in X_n$ (resp. $M^{\pi}\in X_n^{+}$), and we obtain a group (resp. monoid) right action of $S_n$ on $X_n$. 
Thus, we have semi-direct products $X_n\rtimes S_n$ (as a group) and $X_n^{+}\rtimes S_n$ (as a monoid) using the above right action of $S_n$ on $X_n$ and $X^{+}$. \par 

For any two $n$-braids $a$ and $b \in B_n$, it holds that $C(ab)=C(a)+C(b)^{\pi_a}$ (see \cite{Bu, Gu}). 
Thus, the map $C: B_n \to X_n$ induces a group homomorphism $\widetilde{C}: B_n\to X_n\rtimes S_n$ defined by $\widetilde{C}(b)=(C(b), \rho_b)$ for $b\in B_n$, where $\rho_b\in S_n$ is the braid permutation of $b$. 
As noted in \cite{Gu}, ${\rm Ker}\ \widetilde{C}=[P_n, P_n]\times \{id\}$, where $[P_n, P_n]$ is the comutator subgroup of $P_n$. 
On the other hand, for two $n$-braid diagrams $A, B \in {\mathcal B}_n$, it holds that $U(AB)=U(A)+U(B)^{\pi_A}$ (see \cite{AY}). 
Then for each braid diagram $B$ of $b\in B_n$ and each $W\in \{C,N\}$, the map $W: {\mathcal B}_n\to X_n^{+}$ induces a monoid homomorphism $\widetilde{W}: {\mathcal B}_n\to X_n^{+}\rtimes S_n$ defined by $\widetilde{W}(B)=(W(B), \rho_b)$. 
For each $W\in \{C, N\}$, it obviously holds that $\widetilde{W}(B)=(O, id)$ for $B\in {\mathcal B}_n$, where $O$ is the zero matrix, if and only if $B$ is a ``trivial braid diagram'', that is, there are no crossings in $B$.

\subsection{Realizable matrices}
\label{subs-adm}

We say that a matrix $M$ is {\it OU-realizable} (resp.~{\it CN-realizable}) if there exists a braid diagram $B$ such that $U(B)=M$ (resp.~$N(B)=M$). 
We say that a matrix $M$ is {\it C-realizable} if there exists a braid $b$ such that $C(b)=M$. 
By Proposition \ref{prop-adequate}, we have the following proposition. 

\begin{proposition}
A square matrix $M$ is not OU-realizable if $M+M^T$ is not T0. 
\end{proposition}

\begin{example}
For  
\begin{align*}
A=
\begin{pmatrix}
0 & 2 & 0 & 0 \\
1 & 0 & 0 & 1 \\
1 & 0 & 0 & 0 \\
0 & 1 & 2 & 0 
\end{pmatrix}, \ B=
\begin{pmatrix}
0 & 0 & 1 & 1 \\
0 & 0 & 0 & 0 \\
0 & 1 & 0 & 0 \\
1 & 0 & 1 & 0 
\end{pmatrix},
\end{align*}
the matrix $A$ is OU-realizable because $A=U(B_1)$ in Figure \ref{fig-ex-OUCN}, whereas $B$ is not OU-realizable since $B+B^T$ is not T0. 
In addition, $A$ and $B$ are not CN-realizable because they are not symmetric. 
\end{example}

\begin{proposition}
If an $n \times n$ symmetric matrix $M$ is CN-realizable, then any non-negative integer matrix $M_1$ that satisfies $M_1+ M_1^T =M$ is OU-realizable. 
\label{prop-CNOU-ad}
\end{proposition}

\begin{proof}
Let $M$ be a CN-realizable $n \times n$ matrix and let $B$ be an $n$-braid diagram such that $N(B)=M$. 
Let $M_1$ be a matrix such that $M_1 +M_1^T =M$, that is, $M_1 (i,j)+M_1(j,i)=M(i,j)$, and $M_1(i,j) \geq 0$ for any $i$ and $j$. 
For each pair of strands $s_i$ and $s_j$ of $B$, apply crossing changes if necessary so that $s_i$ is over $s_j$ at $M_1(i,j)$ crossings and $s_i$ is under $s_j$ at $M_1(j,i)$ crossings. 
Then we obtain a braid diagram $B'$ such that $U(B')=M_1$. 
\end{proof}

\noindent Since a crossing change does not change the braid permutation, we have the following corollary from the proof of Proposition \ref{prop-CNOU-ad}.

\begin{corollary}
If an $n \times n$ symmetric matrix $M$ is CN-realizable by a pure $n$-braid diagram, then any non-negative integer matrix $M_1$ that satisfies $M_1+ M_1^T =M$ is OU-realizable by a pure $n$-braid diagram. 
\label{cor-oucn}
\end{corollary}

\noindent The following proposition allows us to extend the size of CN-realizable matrices.

\begin{proposition}
Let $M_1$ be a CN-realizable $n \times n$ matrix. 
Let $M_2$ be an $(n+1) \times (n+1)$ matrix such that $M_2(i,j)=M_1(i,j)$ for each $1 \leq i, j \leq n$ and $M_2(i, n+1)=M_2(n+1,j)=0$ for each $1\leq i,j \leq n$. 
Then $M_2$ is also CN-realizable. 
\label{prop-plus1}
\end{proposition}

\begin{proof}
Let $B_1$ be an $n$-braid diagram with $N(B_1)=M_1$. 
Let $B_2$ be an $(n+1)$-braid diagram obtained from $B_1$ by adding a straight strand on the right-hand side. 
Then, $N(B_2)=M_2$. 
\end{proof}

\section{$(0,2)$-CN matrix of pure braid diagram}
\label{section-02pure}

From now on in Sections \ref{section-02pure} and \ref{section-5b}, we denote the CN matrices, which are zero-diagonal symmetric matrices, by strictly upper triangular matrices by replacing the entries below the main diagonal with $0$, for simplicity. 
For example, we denote the CN matrix 
\begin{align*}
\begin{pmatrix}
0 & 2 & 0 & 2 \\
2 & 0 & 2 & 0 \\
0 & 2 & 0 & 2 \\
2 & 0 & 2 & 0 
\end{pmatrix}
\text{ by }
\begin{pmatrix}
0 & 2 & 0 & 2 \\
0 & 0 & 2 & 0 \\
0 & 0 & 0 & 2 \\
0 & 0 & 0 & 0 
\end{pmatrix}.
\end{align*}

\begin{definition}
We call a matrix whose entries are $0$ or $2$ a {\it $(0,2)$-matrix}. 
In particular, when a CN matrix is a $(0,2)$-matrix, we call it a $(0,2)$-CN matrix. 
\end{definition}

In this section, we explore $(0,2)$-CN matrices. 
In Section \ref{sub02-properties}, we see some properties of the $(0,2)$-CN matrices. 
In Section \ref{sub02-ladder}, we introduce a tool to construct a braid diagram that realizes the $(0,2)$-CN matrix for some $(0,2)$-matrices.

\subsection{Properties of $(0,2)$-CN matrix}
\label{sub02-properties}

We have the following proposition from Proposition \ref{prop-even-pure}. 

\begin{proposition}
When the CN matrix $N(B)$ of a braid diagram $B$ is a $(0,2)$-matrix, $B$ is a pure braid diagram. 
\label{prop-02pure}
\end{proposition}

\noindent Then we have the following proposition. 

\begin{proposition}
If $n \times n$ $(0,2)$-matrices $M_1$ and $M_2$ are CN-realizable, then the sum $M_1 +M_2$ is also CN-realizable. 
\label{prop-block}
\end{proposition}

\begin{proof}
Let $B_1$, $B_2$ be $n$-braid diagrams such that $N(B_1)=M_1$, $N(B_2)=M_2$. 
Since $M_1$, $M_2$ are $(0,2)$-matrices, $B_1$, $B_2$ are pure braid diagrams by Proposition \ref{prop-02pure}. 
By Proposition \ref{prop-N-sum}, then, $N(B_1 B_2)=N(B_1)+N(B_2)=M_1 +M_2$. 
Hence, $M_1 +M_2$ is CN-realizable, too. 
\end{proof}

\noindent The following proposition implies that studying $(0,2)$-CN matrices is essential for the study of CN matrices of pure braid diagram. 

\begin{proposition}
Let $M$ be an $n \times n$ strictly upper triangular matrix whose entries are non-negative even numbers. 
Let $M^{02}$ be the $(0,2)$-matrix obtained from $M$ by replacing all positive entries with 2. 
If $M^{02}$ is CN-realizable, then $M$ is also CN-realizable by a pure braid diagram. 
\label{prop-M02}
\end{proposition}

\begin{proof}
Let $B$ be an $n$-braid diagram such that $N(B)=M^{02}$. 
We note that $B$ is a pure braid diagram by Proposition \ref{prop-02pure}.  
For each entry $M(i,j)$, if $M(i,j) \neq M^{02}(i,j)$, that is, $M(i,j) \geq 4$, replace one of the crossings between $s_i$ and $s_j$ of $B$ by $M(i,j)-1$ half twists on $B$, as shown in Figure \ref{fig-replace}. 
We note that this replacememt does not change the braid permutation because $M(i,j)-1$ is also an odd number. 
Thus, we obtain another pure braid diagram $B'$, which satisfies $N(B')=M$. 
\end{proof}
\begin{figure}[ht]
\centering
\includegraphics[width=2cm]{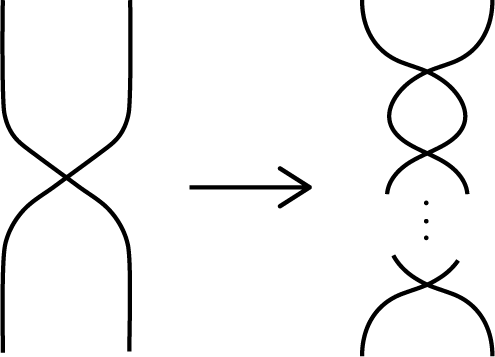}
\caption{Replacement of a crossing with half twists.}
\label{fig-replace}
\end{figure}

\subsection{BW-ladder diagram}
\label{sub02-ladder}

In this subsection, we introduce the ``BW-ladder diagram'' to discuss the CN-realizablity for $(0,2)$-matrices. 
We call a filled vertex a {\it black vertex} and an open vertex a {\it white vertex}. 
A {\it BW-ladder diagram} is a ladder-fashioned diagram consisting of some vertical segments and two types of horizontal edges, a {\it black edge} and a {\it white edge} which have black vertices and white vertices at the endpoints on vertical segments, respectively. 
In particular, we call a BW-ladder diagram with only black (resp.~white) edges a {\it B-ladder diagram} (resp.~{\it W-ladder diagram}). 
We consider white edges whose endpoints are on vertical segments which are next to each other only. 
We denote a black edge that has black vertices at the $i^{th}$ and $j^{th}$ vertical segments ($i<j$) from the left-hand side by $B^i_j$, as shown in Figure \ref{fig-BW}. 
Similarly, we denote a white edge that has white vertices at the $i^{th}$ and $(i+1)^{th}$ vertical segments by $W^i_{i+1}$. 
Each BW-ladder diagram can be represented by a word of $B^i_j$ and $W^i_{i+1}$. \\
\begin{figure}[ht]
\centering
\includegraphics[width=5cm]{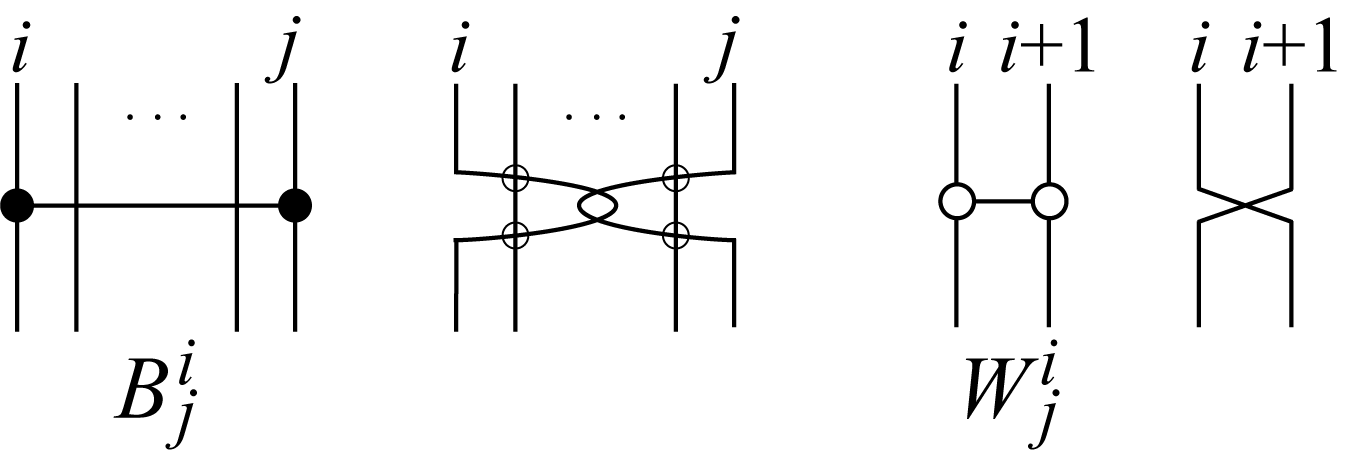}
\caption{A black edge $B^i_j$ can be considered as a hook. A white edge $W^i_{i+1}$ can be considered as a half twist. Undesired crossings are marked with a small circle. }
\label{fig-BW}
\end{figure}
For an $n \times n$ strictly upper triangular $(0,2)$-matrix $M$, a {\it B-ladder diagram of $M$} is a B-ladder diagram with $n$ vertical segments that has a black edge $B^i_j$ if $M(i,j)=2$. 
Considering $B^i_j$ as a hook of the two segments where $B^i_j$ has roots, as shown in the left-hand side in Figure \ref{fig-BW}, we can restore $M$ from the B-ladder diagram. 
Moreover, considering $W^i_{i+1}$ as a half-twist of the segments where $W^i_{i+1}$ has roots, as shown in Figure \ref{fig-BW}, we consider the following local moves.

\begin{definition}
A {\it ladder move} is one of the following moves L1 to L9 on a BW-ladder diagram. (See Figure \ref{fig-l-move}.) 
\begin{itemize}
\item[(L1)] $B^i_{i+1} \leftrightarrow W^i_{i+1} W^i_{i+1}$  for any $i$. 
\item[(L2)] $B^i_j B^k_l \leftrightarrow B^k_l B^i_j$ for any $i<j, \ k<l$. 
\item[(L3)] $W^k_{k+1} W^l_{l+1} \leftrightarrow W^l_{l+1} W^k_{k+1}$ when $k+1<l$ or $l+1<k$. 
\item[(L4)] $W^i_{i+1} W^{i+1}_{i+2} W^i_{i+1} \leftrightarrow W^{i+1}_{i+2} W^i_{i+1} W^{i+1}_{i+2}$ for any $i$. 
\item[(L5)] $B^i_j W^k_{k+1} \leftrightarrow W^k_{k+1} B^i_j$ when $j<k$, $k+1<i$ or $i<k<k+1<j$.
\item[(L6)] $B^i_j W^i_{i+1} \leftrightarrow W^i_{i+1} B^{i+1}_j$ when $i+1 < j$.
\item[(L7)] $W^i_{i+1} B^i_j \leftrightarrow B^{i+1}_j W^i_{i+1}$ when $i+1 < j$.
\item[(L8)] $B^i_j W^{j-1}_j \leftrightarrow W^{j-1}_j B^i_{j-1}$ when $i < j-1$.
\item[(L9)] $W^{j-1}_j B^i_j \leftrightarrow B^i_{j-1} W^{j-1}_j$ when $i < j-1$.
\end{itemize}
\begin{figure}[ht]
\centering
\includegraphics[width=13cm]{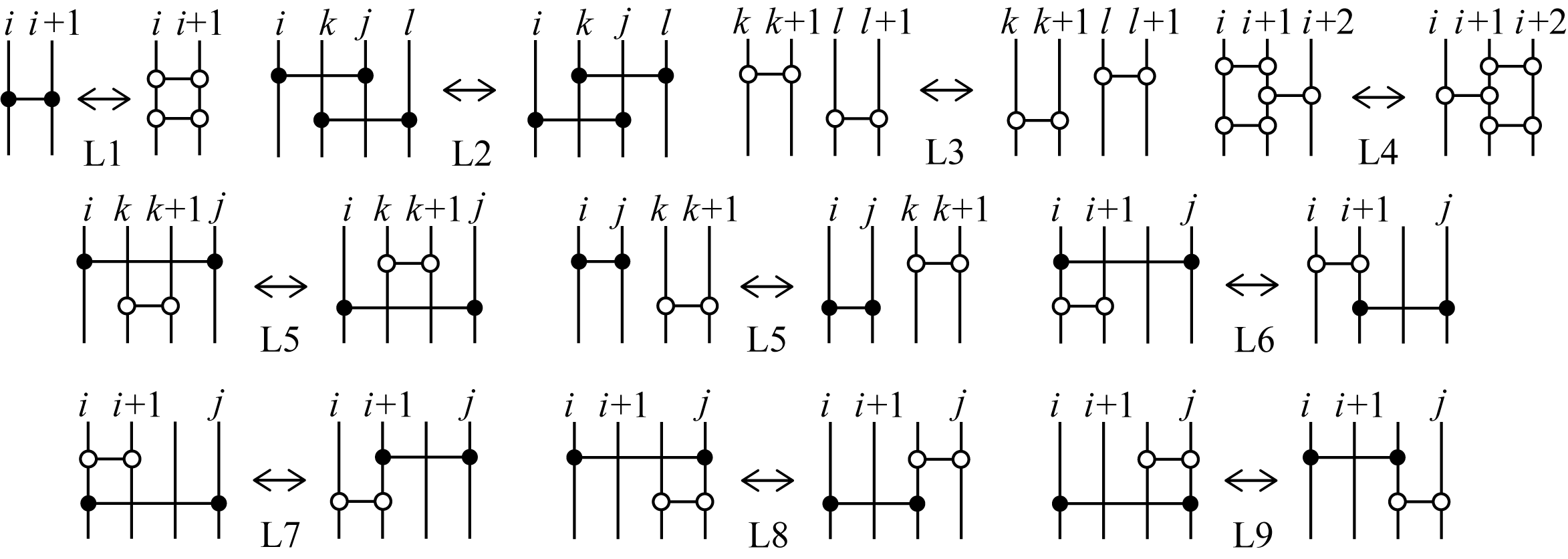}
\caption{Ladder moves.}
\label{fig-l-move}
\end{figure}
\end{definition}

\noindent The ladder move L1 means the replacement of a hook of two segments that are next to each other and a pair of half-twists between them. 
L2 means the change of the order of the black edges. 
L3 and L4 mean the transformations based on the braid relation. 
L5 means the change of the order of a black edge and a white edge whose endpoints are all different. 
L6 to L9 mean the transformation that a root of a hook moves along a half-twist, as shown in Figure \ref{fig-l6}.
\begin{figure}[ht]
\centering
\includegraphics[width=7.5cm]{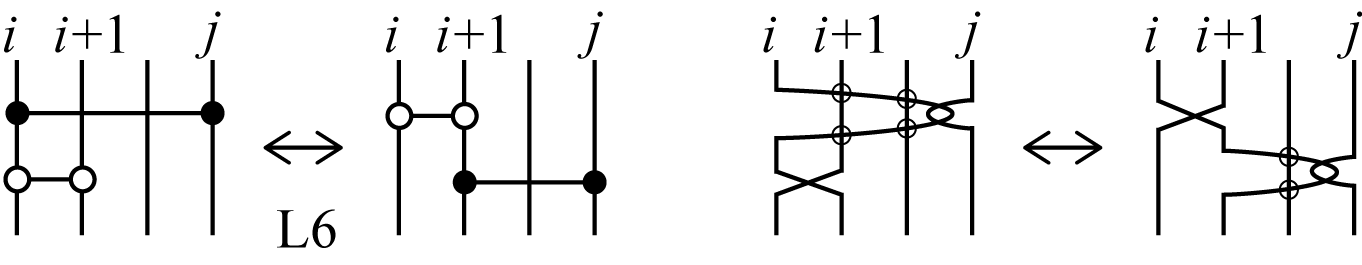}
\caption{The ladder move $L6$.}
\label{fig-l6}
\end{figure}
Since a W-ladder diagram implies a braid diagram (without crossing information), we have the following proposition. 

\begin{proposition}
A strictly upper triangular $(0,2)$-matrix $M$ is CN-realizable if a B-ladder diagram of $M$ can be transformed into a W-ladder diagram by a finite sequence of ladder moves.
\end{proposition}
 
\begin{example}
The $5 \times 5$ $(0,2)$-matrix shown in the left-hand side in Figure \ref{fig-l-ex} is CN-realizable because the B-ladder diagram can be transformed into a W-ladder diagram by ladder moves. 
\begin{figure}[ht]
\centering
\includegraphics[width=13cm]{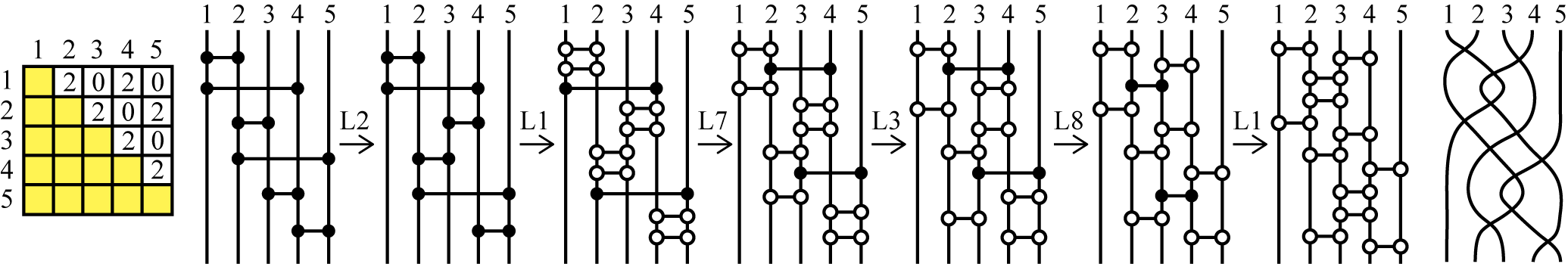}
\caption{A sequence of ladder moves from a B-ladder diagram to a W-ladder diagram. }
\label{fig-l-ex}
\end{figure}
\label{ex-D}
\end{example}

\noindent We have the following proposition. 

\begin{proposition}
The BW-ladder diagram $B^{l-1}_{l} B^{l-2}_{l} B^{l-3}_{l} \dots B^{k+1}_{l} B^k_l$ can be transformed into $W^{l-1}_{l} W^{l-2}_{l-1} W^{l-3}_{l-2} \dots W^{k+1}_{k+2} W^{k}_{k+1} W^{k}_{k+1} W^{k+1}_{k+2} \dots W^{l-3}_{l-2} W^{l-2}_{l-1} W^{l-1}_{l}$ 
by ladder moves. 
\label{prop-bw}
\end{proposition}

\begin{proof}
On $B^{l-1}_{l} B^{l-2}_{l} B^{l-3}_{l} \dots B^{k+1}_{l} B^k_l$, replace $B^{l-1}_l$ with $W^{l-1}_l W^{l-1}_l$ by the ladder move L1. 
Then $W^{l-1}_l W^{l-1}_l B^{l-2}_l B^{l-3}_l \dots B^{k+1}_l B^{k}_l$ is transformed into \\
$W^{l-1}_l B^{l-2}_{l-1} B^{l-3}_{l-1} \dots B^{k+1}_{l-1} B^{k}_{l-1} W^{l-1}_l$ by applying L9 repeatedly. 
Replace $B^{l-2}_{l-1}$ with $W^{l-2}_{l-1} W^{l-2}_{l-1}$ by L1. 
Then we have $W^{l-1}_l W^{l-2}_{l-1} B^{l-3}_{l-2} \dots B^{k+1}_{l-2} B^{k}_{l-2} W^{l-2}_{l-1} W^{l-1}_l$ by L9. 
Repeat this procedure until $B^k_{k+1}$ is replaced with $W^k_{k+1} W^k_{k+1}$. 
(See Figure \ref{fig-lad}.)
\end{proof}
\begin{figure}[ht]
\centering
\includegraphics[width=9cm]{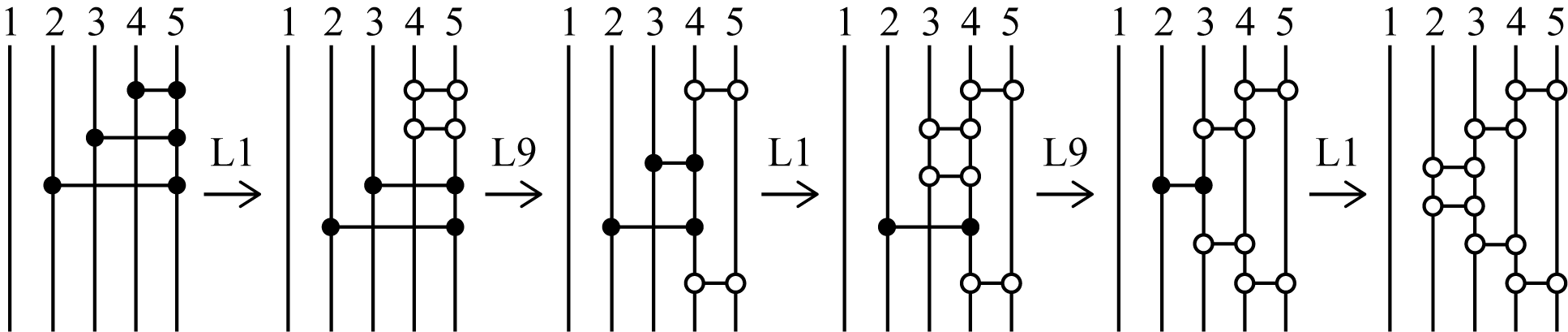}
\caption{A sequence from $B^4_5 B^3_4 B^2_5$ to $W^4_5 W^3_4 W^2_3 W^2_3 W^3_4 W^4_5$.}
\label{fig-lad}
\end{figure}

\noindent From Proposition \ref{prop-bw}, we have the following corollary. 

\begin{corollary}
Let $M$ be an $n \times n$ $(0,2)$-matrix satisfying the following condition for some $k$ with $1 \leq k \leq n$: 
\begin{align*}
& M(i,j)=0 \text{ for } 1 \leq j \leq n-1 \\
& M(i,n)= 
\left\{
\begin{array}{ll}
0 \text{ for } 1 \leq i \leq k-1 \text{ or } i=n \\
2 \text{ for } k \leq i \leq n-1
\end{array}
\right.
\end{align*}
Then $M$ is CN-realizable. 
\label{cor-column}
\end{corollary}

\begin{proof}
The matrix $M$ has the B-ladder diagram of Proposition \ref{prop-bw} with $l=n$, which can be transformed into a W-ladder diagram by ladder moves. 
\end{proof}

\noindent From Propositions \ref{prop-plus1}, \ref{prop-block} and Corollary \ref{cor-column}, we have the following corollary. 

\begin{corollary}
Let $M_1$ be a CN-realizable $n \times n$ strictly upper triangular $(0,2)$-matrix. 
Then, an $(n+1) \times (n+1)$ matrix $M$ satisfying the following condition for some $k$ with $1 \leq k \leq n$ is also CN-realizable. 
\begin{align*}
& M(i,j)=M_1(i,j) \text{ for } 1 \leq j \leq n \\
& M(i,n+1)= 
\left\{
\begin{array}{ll}
0 \text{ for } 1 \leq i \leq k-1 \text{ or } i=n+1 \\
2 \text{ for } k \leq i \leq n
\end{array}
\right.
\end{align*}
\label{cor-n+1}
\end{corollary}

\begin{example}
The $6 \times 6$ matrix 
\begin{align*}
M=
\begin{pmatrix}
0 & 2 & 0 & 2 & 0 & 0 \\
0 & 0 & 2 & 0 & 2 & 0 \\
0 & 0 & 0 & 2 & 0 & 2 \\
0 & 0 & 0 & 0 & 2 & 2 \\
0 & 0 & 0 & 0 & 0 & 2 \\
0 & 0 & 0 & 0 & 0 & 0 \\
\end{pmatrix}
\end{align*}
is CN-realizable because the $5 \times 5$ matrix from the first to the fifth rows and columns is CN-realizable as shown in Figure \ref{fig-l-ex} and the sixth column meets the condition of Corollary \ref{cor-n+1} with $n=5$, $k=3$. 
\end{example}

\noindent We also have the following proposition.

\begin{proposition}
Any $n \times n$ T0 upper triangular $(0,2)$-matrix with $M(i,j)=0$ for $j-i \geq 3$ is CN-realizable. 
\label{prop-ji3}
\end{proposition}

\begin{proof}
Suppose the B-ladder diagram with the following order of black edges, 
$$B^1_2 B^1_3 B^2_3 B^2_4 B^3_4 \dots B^{k}_{k+1} B^{k}_{k+2} B^{k+1}_{k+2} \dots B^{n-3}_{n-1} B^{n-2}_{n-1} B^{n-2}_{n} B^{n-1}_{n},$$
where we ignore $B^i_j$ if $M(i,j)=0$. 
By the ladder move L1, we have 
$$W^1_2 W^1_2 B^1_3 W^2_3 W^2_3 B^2_4 B^3_4 \dots W^{k}_{k+1} W^{k}_{k+1} B^{k}_{k+2} W^{k+1}_{k+2} W^{k+1}_{k+2} \dots B^{n-2}_{n} W^{n-1}_{n} W^{n-1}_{n}.$$
Since $M$ is T0, if $B^k_{k+2}$ exists in the sequence, namely $M(k, k+2)=2$, then at least one of $W^k_{k+1}$ or $W^{k+1}_{k+2}$ exists in the sequence. 
Use the white edge, which is next to $B^k_{k+2}$, to apply L7 or L8 and obtain $B^{k+1}_{k+2}$ or $B^k_{k+1}$. 
By L1, then, we obtain a W-ladder diagram. 
\end{proof}

\noindent We have the following corollary. 

\begin{corollary}
Any $n \times n$ $(0,2)$-matrix $M$ with $M(i,j)=0$ for $j-i \geq 2$ is CN-realizable. 
\label{cor-ji2}
\end{corollary}

\begin{proof}
Since any $n \times n$ matrix $M$ such that $M(i,j)=0$ for $j-i \geq 2$ is T0, it follows from Proposition \ref{prop-ji3}. 
\end{proof}

\section{CN matrices of pure braid diagrams up to $5 \times 5$}
\label{section-5b}

We investigate the $(0,2)$-CN matrices of $2 \times 2$ to $4 \times 4$ in Section \ref{subch-2} and $5 \times 5$ in Section \ref{subch-5}.

\subsection{$(0,2)$-CN matrices of $2 \times 2$ to $4 \times 4$}
\label{subch-2}

For $2 \times 2$ and $3 \times 3$ matrices, we have the following proposition. 

\begin{proposition}
Every $2 \times 2$ or $3 \times 3$ T0 upper triangular $(0,2)$-matrix $M$ is CN-realizable. 
\label{prop-N23}
\end{proposition}

\begin{proof}
It follows from Proposition \ref{prop-ji3}. 
\end{proof}

\noindent For $4 \times 4$ matrices, we have the following proposition. 

\begin{proposition}
Every $4 \times 4$ T0 upper triangular $(0,2)$-matrix $M$ is CN-realizable. 
\label{prop-44}
\end{proposition}

\begin{proof}
By Corollary \ref{cor-n+1} and Proposition \ref{prop-N23}, if a $4 \times 4$ T0 upper triangular $(0,2)$-matrix $M$ has the fourth column as $(0,0,0,0)^T$, $(0,0,2,0)^T$, $(0,2,2,0)^T$ or $(2,2,2,0)^T$, then $M$ is CN-realizable. 
Therefore, it is sufficient to check the following four cases, which is $2^3 -4$, Cases A, B, C, D that the fourth column is $(0,2,0,0)^T$, $(2,0,0,0)^T$, $(2,0,2,0)^T$, $(2,2,0,0)^T$, respectively. 
\begin{itemize}
\item[Case A:] Since the $(1,4)$ entry is $0$, it is CN-realizable by Proposition \ref{prop-ji3}. 

\item[Case B:] By T0, the $(1,2)$ and $(1,3)$ entries are $2$. 
Then the first row will be $(0,2,2,2)$. 
By Propositions \ref{prop-order-rev}, \ref{prop-N23} and Corollary \ref{cor-n+1}, it is CN-realizable. 

\item[Case C:] By T0, the $(1,2)$ entry is $2$. 
Divide the matrix as 
\begin{align*}
\begin{pmatrix}
0 & 2 & \alpha & 2 \\
0 & 0 & 0 & 0 \\
0 & 0 & 0 & 2 \\
0 & 0 & 0 & 0 \\
\end{pmatrix}
+
\begin{pmatrix}
0 & 0 & 0 & 0 \\
0 & 0 & \beta & 0 \\
0 & 0 & 0 & 0 \\
0 & 0 & 0 & 0 \\
\end{pmatrix} ,
\end{align*}
where $\alpha$ and $\beta$ are either $0$ or $2$. 
The first matrix is CN-realizable regardless of the value of $\alpha$, as shown in Figure \ref{fig-4C}. 
The second matrix is also CN-realizable regardless of $\beta$ by Corollary \ref{cor-ji2}. 
Hence, by Proposition \ref{prop-block}, the sum of them is CN-realizable, too. 
\begin{figure}[ht]
\centering
\includegraphics[width=12cm]{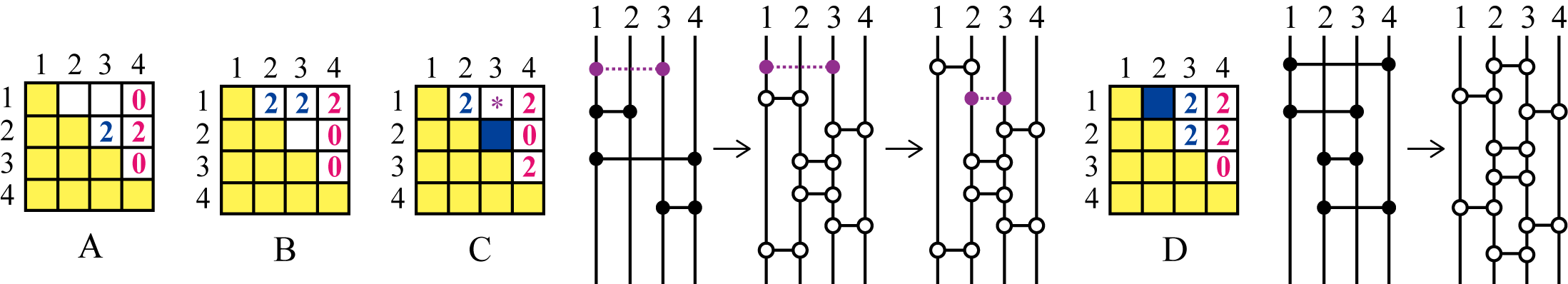}
\caption{Cases A to D. For Case C, ignore the purple-colored broken black edge when the $(1,3)$ entry is $0$.}
\label{fig-4C}
\end{figure}

\item[Case D:] By T0, the $(1,3)$ and $(2,3)$ entries are $2$. 
Divide the matrix as
\begin{align*}
\begin{pmatrix}
0 & 0 & 2 & 2 \\
0 & 0 & 2 & 2 \\
0 & 0 & 0 & 0 \\
0 & 0 & 0 & 0 \\
\end{pmatrix}
+
\begin{pmatrix}
0 & \beta & 0 & 0 \\
0 & 0 & 0 & 0 \\
0 & 0 & 0 & 0 \\
0 & 0 & 0 & 0 \\
\end{pmatrix} ,
\end{align*}
where $\beta$ is $0$ or $2$. 
Since the first matrix is CN-realizable as shown in Figure \ref{fig-4C} and the second matrix is CN-realizable by Corollary \ref{cor-ji2}, the sum is also CN-realizable by Proposition \ref{prop-block}. 
\end{itemize}
\end{proof}

\subsection{$(0,2)$-CN matrices of $5 \times 5$}
\label{subch-5}

In this subsection, we show the following proposition. 

\begin{proposition}
Every $5\times 5$ T0 upper triangular $(0,2)$-matrix $M$ is CN-realizable. 
\label{prop-02-5}
\end{proposition}

\noindent To prove Proposition \ref{prop-02-5}, we prepare three lemmas below.

\begin{lemma}
Let $M$ be a $5 \times 5$ T0 upper triangular $(0,2)$-matrix such that $M(i,j)=0$ when $i<I$ or $j>J$ for some $I, J$ with $1 \leq I <J \leq 5$ and $J-I \leq 3$. 
Then $M$ is CN-realizable. 
\label{lem-small}
\end{lemma}

\begin{proof}
Consider the $(J-I+1) \times (J-I+1)$ matrix $M_1$ such that $M_1(i,j)=M(i+I-1, j+I-1)$. 
We note that $M_1$ is T0 because $M$ is T0. 
Since $J-I+1 \leq 4$, $M_1$ is CN-realizable by Propositions \ref{prop-N23} and \ref{prop-44}. 
Let $B_1$ be a $(J-I+1)$-braid diagram such that $N(B_1)=M_1$. 
Add $I-1$ and $5-J$ straight strands to $B_1$ on the left- and right-hand sides, respectively, as shown in Figure \ref{fig-small} (3). 
Thus, we obtain a $5$-braid diagram $B$ such that $N(B)=M$. 
\end{proof}

\begin{figure}[ht]
\centering
\includegraphics[width=7.5cm]{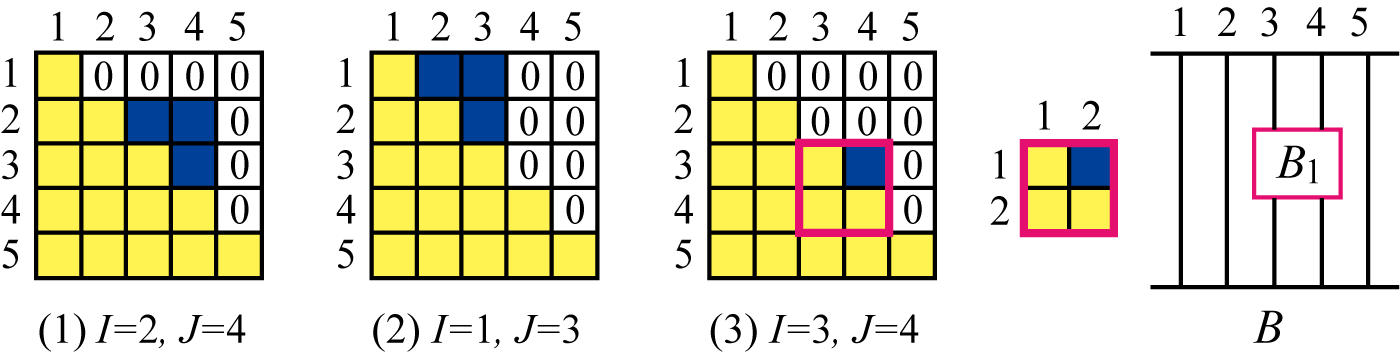}
\caption{(1): The case that $I=2$, $J=4$. (2): The case that $I=1$, $J=3$. (3): The case that $I=3$, $J=4$ and the construction of a braid diagram $B$.}
\label{fig-small}
\end{figure}

\begin{lemma}
If a $5 \times 5$ T0 upper triangular $(0,2)$-matrix $M$ has the configuration of one of $a$ to $h$ in Figure \ref{fig-ab}, then $M$ is CN-realizable. 
\label{lem-ab}
\end{lemma}
\begin{figure}[ht]
\centering
\includegraphics[width=11.2cm]{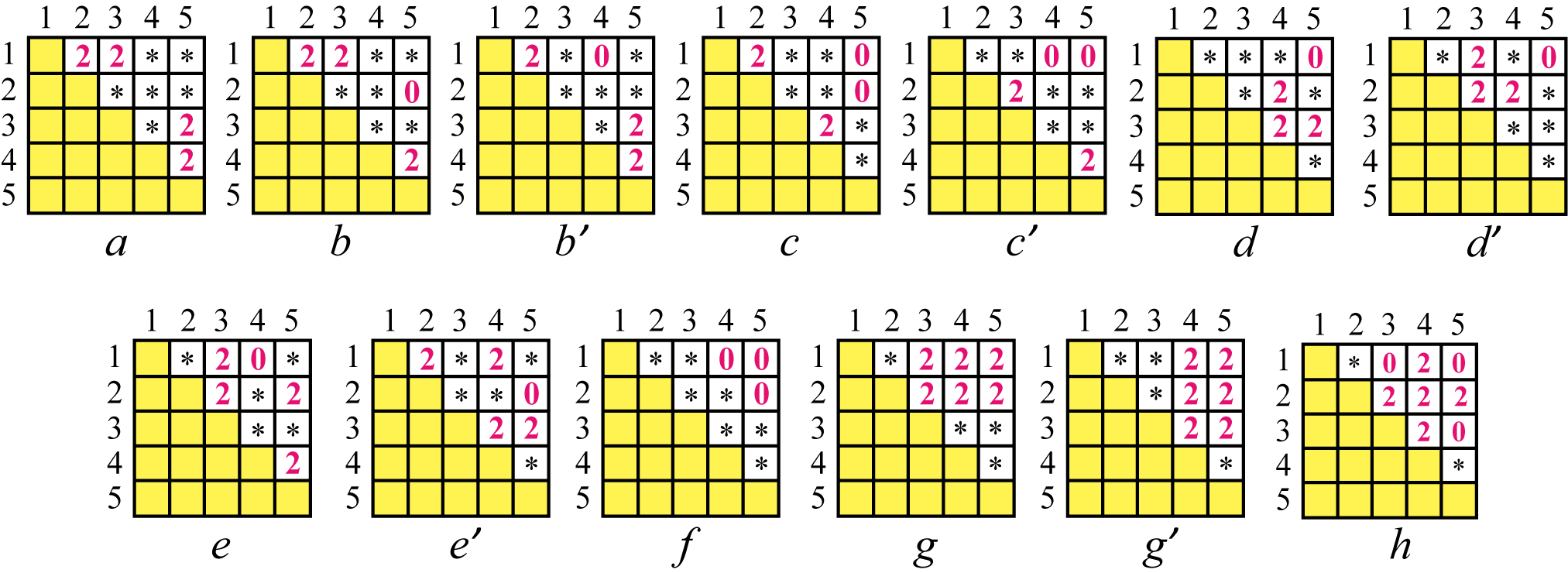}
\caption{The configurations $a$ to $h$ and their reverses. The entries marked $*$ is either $0$ or $2$. }
\label{fig-ab}
\end{figure}

\begin{proof}
We use Proposition \ref{prop-block} and Lemma \ref{lem-small}. 
For the configuration $a$, divide the matrix as 
\begin{align*}
\begin{pmatrix}
0 & 2 & 2 & \alpha_1 & \alpha_2 \\
0 & 0 & 0 & 0 & \alpha_3 \\
0 & 0 & 0 & 0 & 2 \\
0 & 0 & 0 & 0 & 2 \\
0 & 0 & 0 & 0 & 0 
\end{pmatrix}
+
\begin{pmatrix}
0 & 0 & 0 & 0 & 0 \\
0 & 0 & \beta_1 & \beta_2 & 0 \\
0 & 0 & 0 & \beta_3 & 0 \\
0 & 0 & 0 & 0 & 0 \\
0 & 0 & 0 & 0 & 0 
\end{pmatrix} ,
\end{align*}
where $\alpha_i$ and $\beta_j$ are either 0 or 2 ($i, j=1,2,3$). 
The first matrix is CN-realizable as shown in Figure \ref{fig-a} regardless of the values of $\alpha_1$, $\alpha_2$ and $\alpha_3$. 
\begin{figure}[ht]
\centering
\includegraphics[width=6cm]{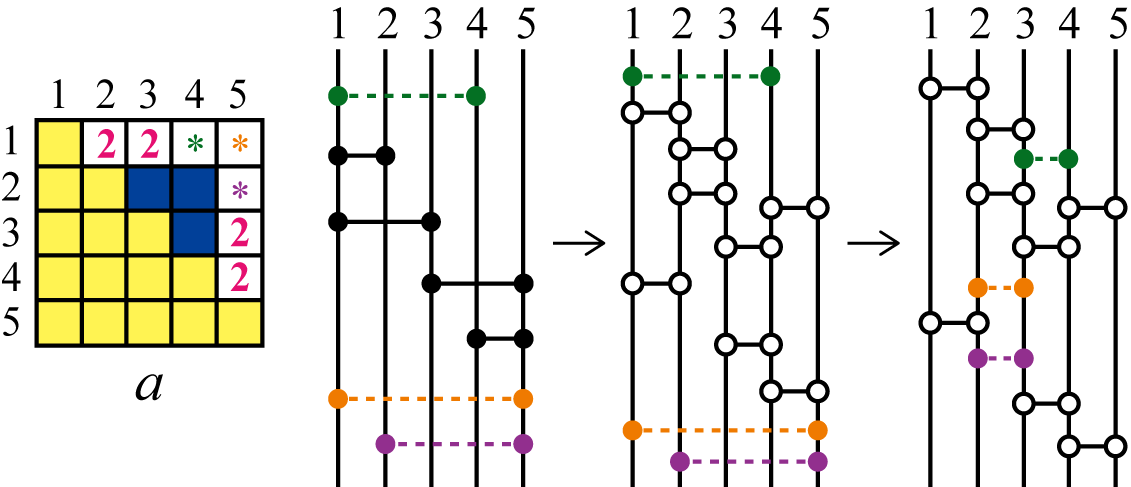}
\caption{Any colored broken black edge $B^i_j$ of $B^1_4$, $B^1_5$ and $B^2_5$ can be ignored when the corresponding $(i,j)$ entry of the matrix is $0$.}
\label{fig-a}
\end{figure}
The second matrix is CN-realizable by Lemma \ref{lem-small}. 
Hence, the sum is also CN-realizable by Proposition \ref{prop-block}. 
In the same way, we can see the CN-realizability for the configurations $b$ to $e$ as shown in Figure \ref{fig-b}. 
We note that we use Lemma \ref{lem-small} and Proposition \ref{prop-block} twice for the configurations $c, d, e$. 
\begin{figure}[ht]
\centering
\includegraphics[width=13cm]{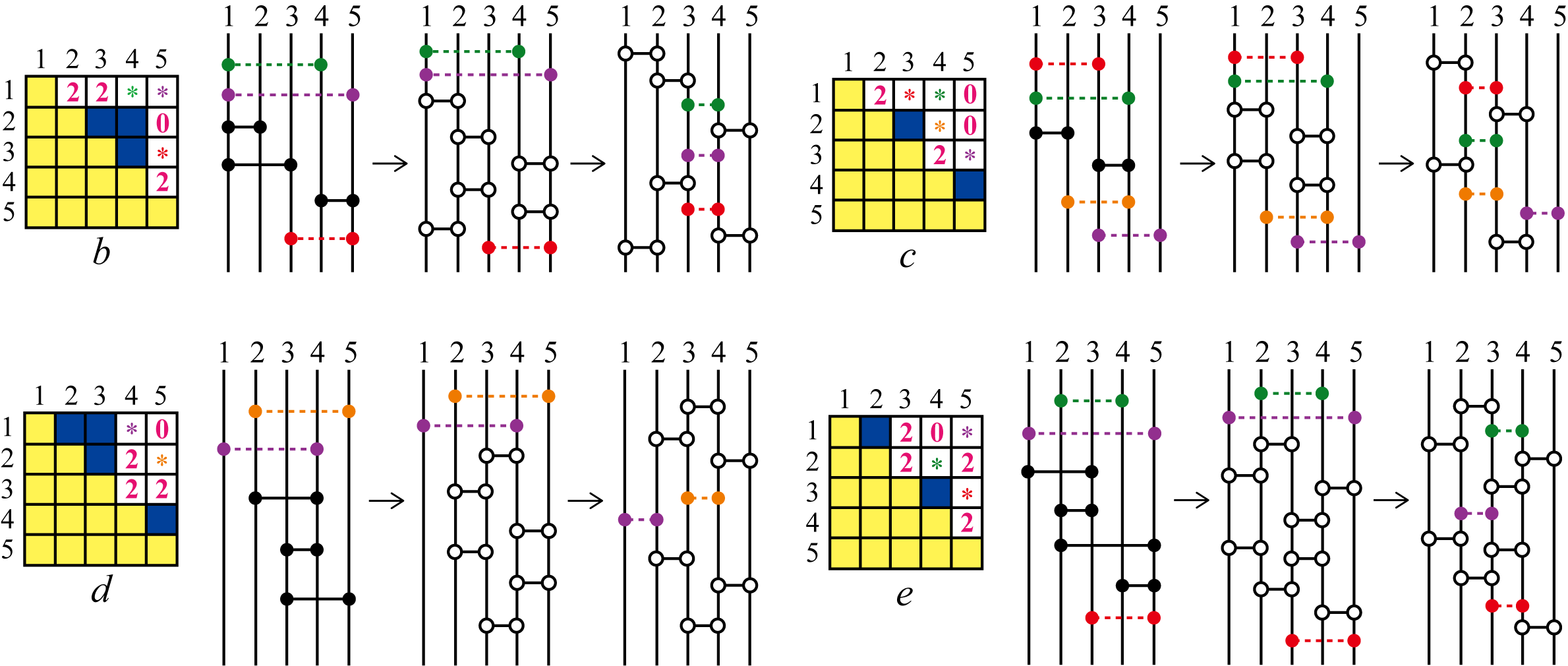}
\caption{Any colored broken black edge $B^i_j$ can be ignored when the corresponding $(i,j)$ entry of the matrix is $0$.}
\label{fig-b}
\end{figure}
The configuration $f$ realizes a CN matrix by Proposition \ref{prop-ji3}. 
For the configurations $g$ and $h$, the matrices 
\begin{align*}
\begin{pmatrix}
0 & 0 & 2 & 2 & 2 \\
0 & 0 & 2 & 2 & 2 \\
0 & 0 & 0 & 0 & 0 \\
0 & 0 & 0 & 0 & 0 \\
0 & 0 & 0 & 0 & 0 \\
\end{pmatrix}
\text{ and }
\begin{pmatrix}
0 & 0 & 0 & 2 & 0 \\
0 & 0 & 2 & 2 & 2 \\
0 & 0 & 0 & 2 & 0 \\
0 & 0 & 0 & 0 & 0 \\
0 & 0 & 0 & 0 & 0 \\
\end{pmatrix}
\end{align*}
are CN-realizable as shown in Figure \ref{fig-h}. 
\begin{figure}[ht]
\centering
\includegraphics[width=7.6cm]{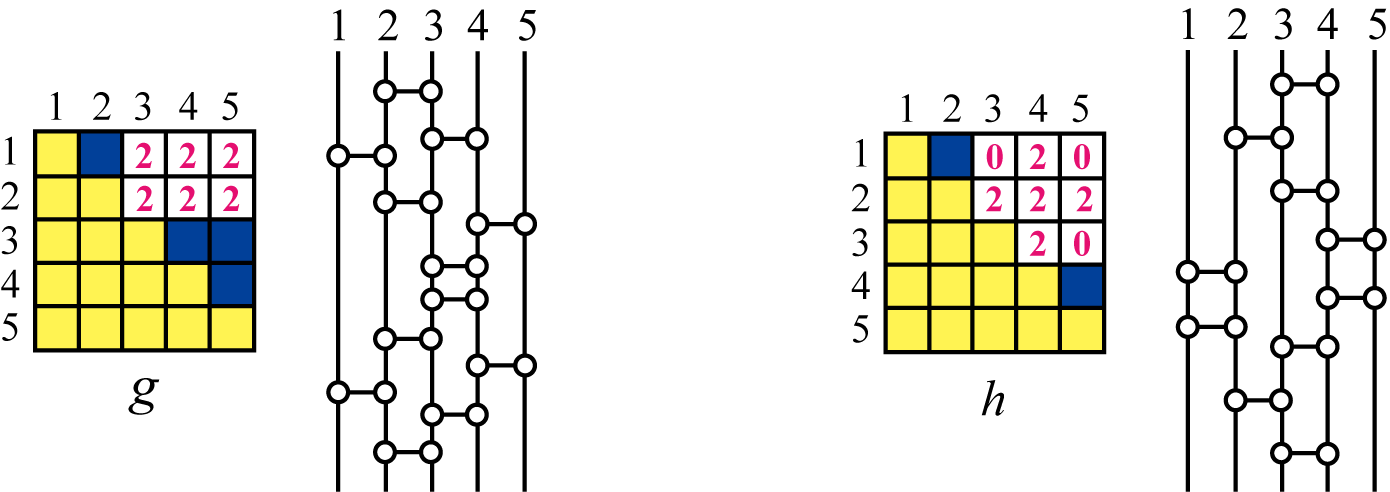}
\caption{The configurations $g$ and $h$ realize CN matrices.}
\label{fig-h}
\end{figure}
Hence, the configurations $g$ and $h$ also realize CN-realizable matrices by Lemma \ref{lem-small} and Proposition \ref{prop-block}. 
\end{proof}

\noindent We note that the reverses $a'$ to $h'$ of the configurations $a$ to $h$ also realize a CN matrix by Proposition \ref{prop-order-rev}. 
We also have the following lemma. 

\begin{lemma}
If a $5 \times 5$ T0 upper triangular $(0,2)$-matrix $M$ has one of the configurations shown in Figure \ref{fig-j}, which we call configuration $i$, then $M$ is CN-realizable. 
\end{lemma}

\begin{proof}
This follows from Proposition \ref{prop-order-rev} and Corollary \ref{cor-n+1}. 
\end{proof}

\begin{figure}[ht]
\centering
\includegraphics[width=8.5cm]{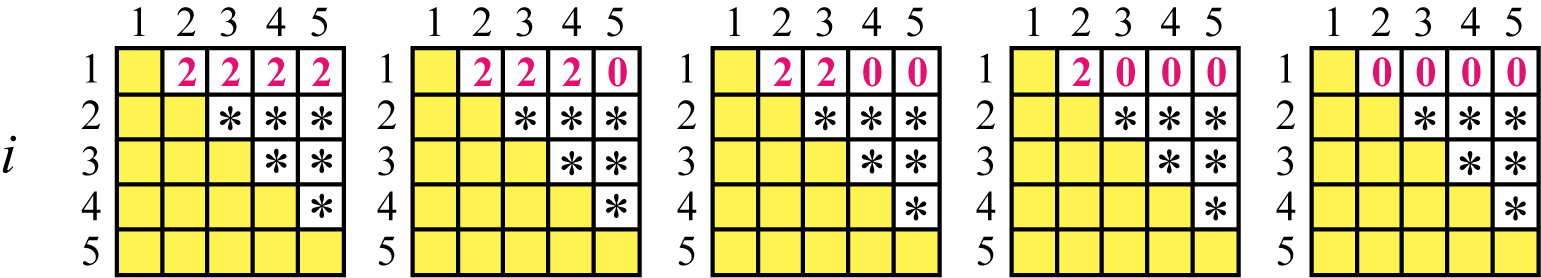}
\caption{The configuration $i$.}
\label{fig-j}
\end{figure}

\noindent Now we prove Proposition \ref{prop-02-5}. \\

\noindent {\it Proof of Proposition \ref{prop-02-5}.} \ 
By Corollary \ref{cor-n+1}, it is sufficient to check the eleven cases, which is $2^4 -5$, on the fifth column A to K in Figure \ref{fig-ABC}. 
By T0, some of the entries in the first to fourth columns are automatically assigned, as shown in Figure \ref{fig-ABC}. 
\begin{figure}[ht]
\centering
\includegraphics[width=10cm]{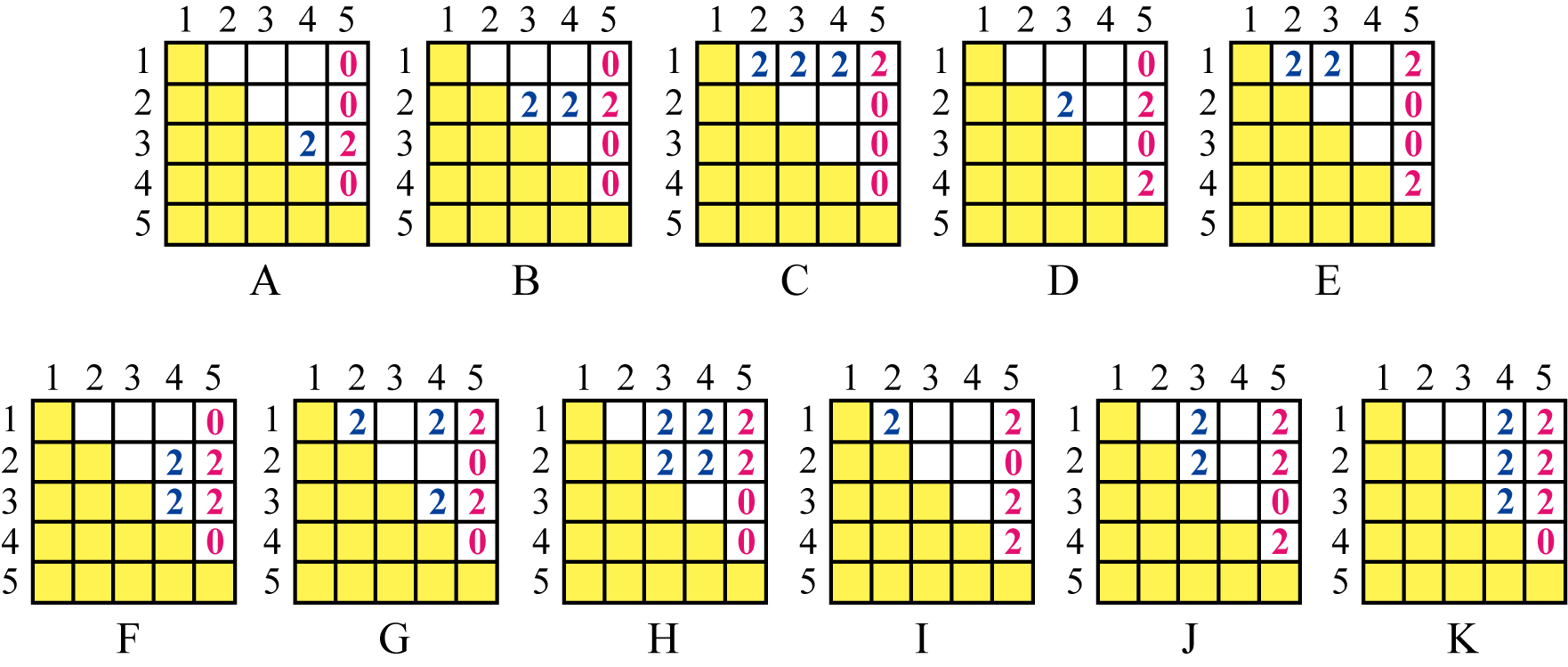}
\caption{Cases A to K.}
\label{fig-ABC}
\end{figure}

\begin{itemize}
\item[Case A:] By the configuration $f$, it is sufficient to suppose that the $(1,4)$ entry is $2$. 
Then, by configuration $c$, suppose that the $(1,2)$ entry is $0$. 
By T0, then, the $(2,4)$ entry is $2$, and observe that it has configuration $d$. 
Hence, any matrix $M$ of Case A is CN-realizable. 

\item[Case B:] By the configuration $d'$, suppose that the $(0,3)$ entry is $0$. 
By configuration $i$, suppose that the $(1,4)$ entry is $2$. 
This has the configuration $h$. 

\item[Case C:] This has the configuration $i$.

\item[Case D:] By configuration $c'$, suppose that the $(1,4)$ entry is $2$. 
Then, by configuration $i$, suppose the three subcases shown in Figure \ref{fig-AEIJ}. 
The first subcase has the configuration $h$. 
The second subcase has the braid diagram as shown in the figure. 
(See Example \ref{ex-D} for the construction of the braid diagram.) 
The third has the configuration $d'$. 

\item[Case E:] By the configuration $i$, suppose that the $(1,4)$ entry is $0$. 
This has the configuration $b$. 

\item[Case F:] This has the configuration $d$. 

\item[Case G:] This has the configuration $e'$. 

\item[Case H:] This has the configuration $g$. 

\item[Case I:] By the configuration $a$, suppose that the $(1,3)$ entry is $0$. 
By the configuration $b'$, suppose that the $(1,4)$ entry is $2$. 
By T0, then, the $(3,4)$ entry is $2$. 
This has the configuration $e'$. 

\item[Case J:] By the configuration $e$, suppose that the $(1,4)$ entry is $2$. 
By the configuration $h$, suppose that the $(2,4)$ entry is $2$. 
Then, it has the configuration $g$. 

\item[Case K:] This has the configuration $g'$. 
\end{itemize}
\begin{figure}[ht]
\centering
\includegraphics[width=12cm]{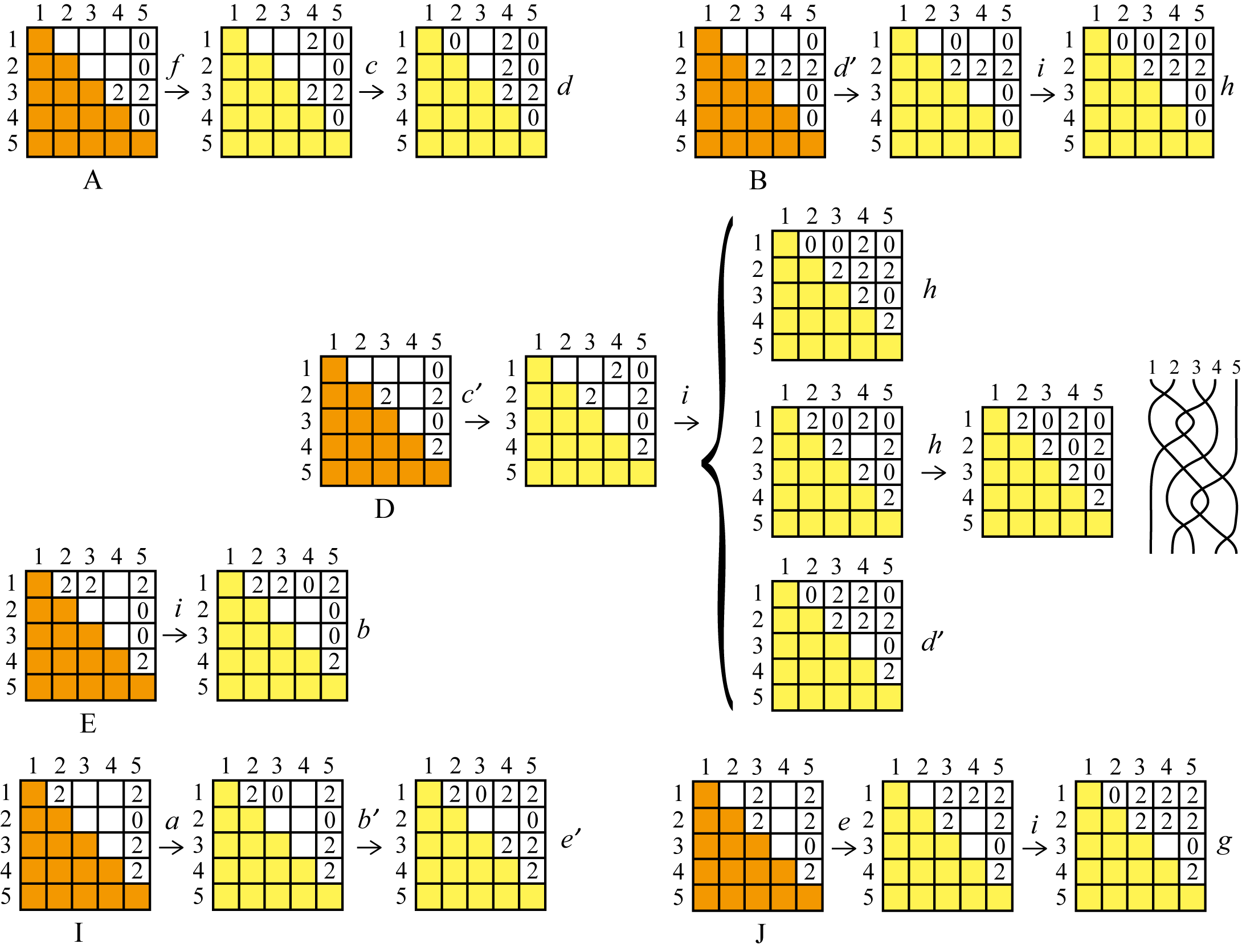}
\caption{Cases A, B, D, E, I and J.}
\label{fig-AEIJ}
\end{figure}
Hence, every $5\times 5$ T0 upper triangular $(0,2)$-matrix is CN-realizable. 
\qed \\

\noindent By Propositions \ref{prop-N23}, \ref{prop-44} and \ref{prop-02-5}, we have the following corollary. 

\begin{corollary}
When $n \leq 5$, any $n\times n$ T0 upper triangular $(0,2)$-matrix is CN-realizable. 
\label{cor-CN5}
\end{corollary}

\noindent We also have the following corollary by Corollary \ref{cor-CN5} and Proposition \ref{prop-M02}.

\begin{corollary}
When $n \leq 5$, any $n\times n$ non-negative even T0 upper triangular matrix is CN-realizable by a pure $n$-braid diagram. 
\label{prop-CN5}
\end{corollary}

\section{Proof of Theorem \ref{thm-OU5} and applications}
\label{section-pf}

We prove Theorem \ref{thm-OU5} in Section \ref{su-pf} and then discuss the crossing matrices of positive pure braids in Section \ref{su-c}, the OU matrices of braid diagrams that are not necessarily pure in Section \ref{su-o}. 
In this section, we consider a CN matrix as a symmetric matrix as defined in Section \ref{section-preliminaries}, not a strictly upper triangular matrix as in Sections \ref{section-02pure} and \ref{section-5b}.

\subsection{Proof of Theorem \ref{thm-OU5}}
\label{su-pf}

In this subsection, we prove Theorem \ref{thm-OU5}. 
We have the following proposition. 

\begin{proposition}
When $n \leq 5$, an $n \times n$ matrix $M$ is the CN matrix of a pure $n$-braid diagram if and only if $M$ is a non-negative even T0 symmetric matrix.
\label{prop-CN5p}
\end{proposition}

\begin{proof}
By Propositions \ref{prop-even-pure} and \ref{prop-adequate}, the CN matrix $N(B)$ of a pure $n$-braid diagram $B$ is a non-negative even T0 symmetric matrix. 
On the other hand, a non-negative even T0 symmetric matrix is CN-realizable by a pure $n$-braid diagram by Corollary \ref{prop-CN5}. 
\end{proof}

\noindent We prove Theorem \ref{thm-OU5}. \\

\noindent {\it Proof of Theorem \ref{thm-OU5}.} \ 
By Propositions \ref{prop-even-pure} and \ref{prop-adequate}, the OU matrix $U(B)$ of any pure $n$-braid diagram $B$ is an $n \times n$ matrix such that $U(B)+(U(B))^T$ is a non-negative even T0 matrix. \par 
Let $M$ be an $n \times n$ zero-diagonal matrix ($n \leq 5$) such that $M+M^T$ is a non-negative even T0 matrix. 
Since $M+M^T$ is CN-realizable by a pure $n$-braid diagram by Corollary \ref{prop-CN5}, $M$ is OU-realizable by a pure $n$-braid diagram by Corollary \ref{cor-oucn}. 
\qed \\

\subsection{C-realizable matrix}
\label{su-c}

In this subsection, we discuss the crossing matrices of positive pure braids. 
We have the following lemmas. 

\begin{lemma}
A zero-diagonal symmetric matrix $M$ is T0 if and only if $M+M^T$ is T0. 
\label{lem-symT0}
\end{lemma}

\begin{proof}
When $M$ is symmetric, i.e., $M(i,j)=M(j,i)$, $M(i,j)=0$ if and only if $M(i,j)+M^T(i,j)=0$. 
\end{proof}

\begin{lemma}
When $n \leq 5$, an $n \times n$ matrix $M$ is the OU matrix of a pure $n$-braid diagram such that each pair of strands $s_i$ and $s_j$ has the same number of crossings such that $s_i$ is over $s_j$ and $s_j$ is over $s_i$ if and only if $M$ is a non-negative integer T0 symmetric matrix. 
\label{lem-sisj}
\end{lemma}

\begin{proof}
By Theorem \ref{thm-OU5}, the OU matrix of a pure $n$-braid $M$ $(n \leq 5)$ satisfies that $M+M^T$ is a non-negative even T0 $n \times n$ matrix. 
Moreover, $M$ is symmetric with the condition of the crossing number of $s_i$ and $s_j$, i.e., $M(i,j)=M(j,i)$. 
Hence $M$ itself is also T0 by Lemma \ref{lem-symT0}. \par 
On the other hand, let $M$ be an $n \times n$ non-negative integer T0 symmetric matrix ($n \leq 5$). 
Then $M+M^T$ is also T0 by Lemma \ref{lem-symT0} and the entries of $M+M^T$ are all non-negative even numbers since $M$ is symmetric. 
Therefore, $M+M^T$ is CN-realizable by a pure $n$-braid diagram by Proposition \ref{prop-CN5p}. 
Then $M$ is OU-realizable by a pure $n$-braid diagram $B$ by Corollary \ref{cor-oucn}. 
Moreover, $B$ is a braid diagram such that each pair of strands $s_i$, $s_j$ of $B$ has the same number of crossings such that $s_i$ is over $s_j$ and $s_j$ is over $s_i$ because $M(i,j)=M(j,i)$. 
\end{proof}

\noindent We have the following corollaries. 

\begin{corollary}
When $n \leq 5$, the crossing matrix $C(B)$ of any positive pure $n$-braid diagram $B$ is an $n \times n$ non-negative integer T0 symmetric matrix.
\label{cor-5ppp}
\end{corollary}

\begin{proof}
It follows from Lemma \ref{lem-sisj} since $C(B)=U(B)$ when $B$ is a positive braid diagram $B$ by Proposition \ref{prop-OU-C}. 
\end{proof}

\begin{corollary}
When $n \leq 5$, any $n \times n$ non-negative integer T0 symmetric matrix is C-realizable by a positive pure $n$-braid. 
\label{cor-5pp}
\end{corollary}

\begin{proof}
For the pure $n$-braid diagram $B$ in the proof of Lemma \ref{lem-sisj} with $U(B)=M$, apply crossing changes if necessary to obtain a positive braid diagram $B'$. 
We note that $B'$ is still a pure $n$-braid diagram such that each pair of strands $s_i$ and $s_j$ has the same number of crossings such that $s_i$ is over $s_j$ and $s_j$ is over $s_i$. 
Then $C(B')=U(B')=U(B)=M$ by Proposition \ref{prop-OU-C} and Lemma \ref{lem-sisj}. 
\end{proof}

\noindent By Corollaries \ref{cor-5ppp} and \ref{cor-5pp}, we have the following corollary. 

\begin{corollary}
Conjecture \ref{conj-C} is true when $n\leq 5$. 
\label{cor-C1}
\end{corollary}

\subsection{OU matrix and CN matrix for a general braid type}
\label{su-o}

A {\it simple braid} (or a {\it permutation braid}) is a braid that has a {\it simple diagram}, that is, each crossing is positive and each pair of strands has at most one crossing, which has been used in studies of the normal form of braids (see \cite{El, Ga, ThB}). 
Let $\Sigma_n$ be the set of simple $n$-braids. 
There is one-to-one correspondence from $\Sigma_n$ to $S_n=Sym\{1,\dots ,n\}$ by taking their braid permutations $\rho$. 
Let $X_n$ be the set of $n \times n$ zero-diagonal integer matrices. 
As explained in \cite{Bu, Gu}, using Lemma 9.1.6 in \cite{ThB} directly, we have the following lemma. 

\begin{lemma}
The map $C|_{\Sigma_n}:\Sigma_n \to X_n$ is injective and its image $C(\Sigma_n)$ is the set of $L\in X_n$ satisfying the following conditions: 
\begin{itemize}
\item[(i)] $L(i,j)=0$ if $i\geq j$, 
\item[(ii)] $L(i,j)\in \{0,1\}$, 
\item[(iii)] $L$ is both T0 and T1, i.e., if whenever $1\leq i<j<k\leq n$, then $L(i,j)=L(j,k)=p$ implies $L(i,k)=p$ for $p=0,1$.
\end{itemize}
\label{lem-iii}
\end{lemma}

\noindent We call a square matrix satisfying the conditions (i)-(iii) in Lemma \ref{lem-iii} a {\it simple matrix}.
Take a simple diagram $B_0$ of $b_0\in \Sigma_n$. 
If a pair of strands $s_i$ and $s_j \ (i < j)$ of $B_0$ has a crossing, then $s_i$ passes under $s_j$, and equivalently $s_j$ passes over $s_i$ at the crossing. 
Thus, $U(B_0)=C(B_0)(=C(b_0))$ for any simple diagram $B_0$ of any $b_0\in \Sigma_n$. 
We note that $N(B_0)=U(B_0)+U(B_0)^T$ and $N(i,j)=U(i,j)$ for any $i, j$. 
By Lemma \ref{lem-iii}, we have the following lemma. 

\begin{lemma}
The map $N|_{\Sigma_n}:\Sigma_n \to X_n$ is injective and its image $N(\Sigma_n)$ is the set of $L\in X_n$ satisfying the following conditions:
\begin{itemize}
\item[(I)] $L(i,j)=L(j,i)$ for any $i, j$, 
\item[(I\hspace{-0.5pt}I)] $L(i,j)\in \{0,1\}$, 
\item[(I\hspace{-0.5pt}I\hspace{-0.5pt}I)] $L$ is T0 and T1.
\end{itemize}
\label{lem-iiii}
\end{lemma}

\noindent We call a matrix satisfying the conditions (I)-(I\hspace{-0.5pt}I\hspace{-0.5pt}I) in Lemma \ref{lem-iiii} a {\it double simple matrix}. 
We also remark that $C(\Sigma_n) \subset X_n^{+}$ and $C(\Sigma_n) \subset X_n^{+}$, where $X_n^{+}$ is the set of matrices $M\in X_n$ such that $M(i,j) \geq 0$ for any $i, j$. 
In conclusion, we have the following corollary. 

\begin{corollary}
Let $n\leq 5$. 
For each braid $b\in B_n$,
\begin{itemize}
\item[(1)] there exists a braid diagram $B$ of $b$ such that $U(B)=M_1+L_1$ for some $M_1, L_1\in X_n^{+}$ such that ${M_1}+M_1^T$ is non-negative even T0 and $L_1$ is simple.
\item[(2)] There exists a braid diagram $B$ of $b$ such that $N(B)=M_2+L_2$ for some $M_2, L_2\in X_n^{+}$ such that $M_2$ is non-negative even T0 symmetric and $L_2$ is double simple.
\end{itemize}
\end{corollary}

\begin{proof}
Let $\rho_b\in S_n$ be the braid permutation of $b$ and let $\rho_b^{+}\in \Sigma_n$ be the permutation $n$-braid corresponding to $\rho_b$. 
Let $a=b(\rho_b^{+})^{-1}\in B_n$. 
It is easily seen that $a\in P_n$. 
Take a pure braid diagram $A$ of $a$ and take a simple diagram $B_0$ of $\rho_b^{+}$. 
Let $L_1=U(B_0),\ L_2=N(B_0)(={L_1}+L_1^T), M_1=U(A)$ and $M_2=N(A)(=M_1+M_1^T)$. 
Then, $U(B)=M_1+L_1$ by Proposition \ref{prop-BC-OU} and $N(B)=M_2+L_2$ by Proposition \ref{prop-N-sum}. 
We can see that $M_1+M_1^T$ is non-negative even T0 by Theorem \ref{thm-OU5} and $M_2$ is non-negative even T0 symmetric. 
We can also see that $L_1$ is simple by Lemma \ref{lem-iii} and $L_2$ is double simple by Lemma \ref{lem-iiii}. 
\end{proof}

\section*{Acknowledgment}
The work of the first author was partially supported by the JSPS KAKENHI Grant Number JP21K03263. 
The work of the second author was partially supported by the JSPS KAKENHI Grant Number JP19K03508.

\end{document}